\documentclass[11pt]{amsart}

\usepackage{a4wide}
\usepackage{amsmath,amssymb,amsthm,amsfonts}
\usepackage{amsrefs}
\usepackage[utf8]{inputenc}
\usepackage[T1]{fontenc}
\usepackage{xcolor}

\def\R {\mathbb{R}}

\def\S {\mathbb{S}}
\def\Z {\mathbb{Z}}
\def\S {\mathbb{S}}
\renewcommand{\phi}{\varphi}
\renewcommand{\epsilon}{\varepsilon}
\newcommand{\eps}{\varepsilon}

\renewcommand{\leq}{\leqslant}
\renewcommand{\le}{\leqslant}
\renewcommand{\geq}{\geqslant}
\renewcommand{\ge}{\geqslant}

\renewcommand{\div}{\mathrm{div}}

\newtheorem{proposition}{Proposition}[section]
\newtheorem{theorem}[proposition]{Theorem}
\newtheorem{corollary}[proposition]{Corollary}
\newtheorem{lemma}[proposition]{Lemma}
\theoremstyle{definition}
\newtheorem{definition}[proposition]{Definition}

\newtheorem{remark}[proposition]{Remark}
\numberwithin{equation}{section}

\title[Anisotropic mean curvature flow of Lipschitz graphs]{Anisotropic  mean curvature flow of Lipschitz graphs and convergence to self-similar solutions} 

\author[A. Cesaroni, H. Kr\"oner, M. Novaga]{}

\subjclass{ 53C44, 
 35K93 
 }

\keywords{Anisotropic mean curvature flow, self-similar solutions, long time behavior.}

 \email{annalisa.cesaroni@unipd.it}
 \email{heiko.kroener@uni-due.de}
 \email{matteo.novaga@unipi.it}

\thanks{The authors were supported by the INDAM-GNAMPA and by the PRIN Project 2019/24 {\it Variational methods for stationary and evolution problems with singularities and interfaces}.}

\begin{document}
\maketitle

\centerline{\scshape Annalisa Cesaroni }
\medskip
{\footnotesize
 \centerline{Department of Statistical Sciences, University of Padova}
   \centerline{Via Cesare Battisti 141, 35121 Padova, Italy}
} 
\medskip

\centerline{\scshape Heiko Kr\"oner}
\medskip
{\footnotesize
 \centerline{Universit\"at Duisburg-Essen, Fakult\"at f\"ur Mathematik}
 \centerline{Thea-Leymann-Stra\ss e 9, 45127, Essen, Germany}
 }
 
\medskip

\centerline{\scshape Matteo Novaga}
{\footnotesize
\centerline{Department of Mathematics, University of Pisa}
   \centerline{Largo Bruno Pontecorvo 5, 56127 Pisa, Italy}
}

\begin{abstract}
We consider the anisotropic  mean curvature flow of entire Lipschitz graphs. We  
prove existence and uniqueness of expanding  self-similar solutions which are asymptotic to a prescribed cone, and we
characterize the long time behavior of solutions, after suitable rescaling,
when the initial datum is a sublinear perturbation of a  cone. 
In the case of regular anisotropies, we prove the stability of self-similar solutions asymptotic to
strictly mean convex cones, with respect to perturbations vanishing at infinity. We also show the
stability of hyperplanes, with a proof which is novel also for the isotropic mean curvature flow.
\end{abstract}

\tableofcontents

\section{Introduction}
We consider the evolution of sets $t\mapsto E_t$ in $\R^{N+1}$ governed by the geometric law 
 \begin{equation}\label{mcf} \partial_t p\cdot \nu(p)= -\psi(\nu(p))H_\phi(p, E_t),\end{equation} 
 where $\nu(p)$ is the  exterior normal  at
 $p\in \partial E_t$, $\psi$ is  a positive, continuous, $1$-homogeneous function representing the mobility,  
 $\phi$ is a norm  representing the surface tension, and $H_\phi(p)$  is the anisotropic mean curvature of 
 $\partial E_t$ at $p$, associated with $\phi$, see Definition \ref{defanisotropy}. This evolution is an analogue of the classical (isotropic) mean curvature flow, 
 which corresponds to the case $\phi(x)=\psi(x)=|x|$ and it is studied  as model of crystal growth, see \cites{gp1, gp2, cmnp1, cmnp2}. 
 Existence and uniqueness of the level set flow associated to \eqref{mcf} have been obtained  for general  mobilities $\psi$ and purely crystalline norms $\phi$ in \cites{gp1,gp2}, 
  in the viscosity setting, whereas the case of general norms $\phi$ with convex  mobilities $\psi$  has been treated in \cites{cmnp1, cmnp2},
  in the distributional setting. 
 
 In this paper we consider  the evolution  of subgraphs of entire Lipschitz functions, in the  case in which either $\phi$ is regular (see \eqref{phiregular}) or  $\psi$ is a norm.  
 In particular, we will assume that there exists a Lipschitz continuous function  $u_0:\R^N\to \R$ such that the initial set $E_0$ coincides with $\{(x,z)\ |\ z\leq u_0(x)\}$. 
By monotonicity of the flow and invariance with respect to translations,  the evolution $E_t$ is  defined for all times and is the subgraph of a Lipschitz function, that is, 
$E_t=\{(x,z)\ |\ z\leq u(x,t)\}$. In  Section \ref{secass} we describe the main   properties  of this flow. 

Since the evolution is defined for all times, we are interested  in the analysis of the long time patterns of the evolution. We recall that evolution of entire Lipschitz graphs in the isotropic setting has been considered in  \cite{eh}, see also  \cite{clutt},  whereas the case of fractional mean curvature flow has been considered  recently by two of the authors in \cite{cn}.

As in the isotropic case,  the long time attractors of the flow starting from  entire Lipschitz graphs  are self-similar expanding solutions, defined as follows: 
\begin{definition} \label{defhom} 
An  expanding homothetic solution  is a solution    to \eqref{mcf} such that   $E_t= \lambda(t)E_1$ where $\lambda(1)=1$ and  $\lambda'(t)\geq 0$ for $t>1$. 
This is equivalent to assume that  $E_1$ is a solution to \begin{equation}\label{exp} c(  p\cdot \nu(p))=-\psi(\nu(p))  H_\phi(p  , E_1)\end{equation}  for some $c\geq 0$.  

  \end{definition} 
 
In Section \ref{sectionhom}  we characterize the graphical expanding solutions as evolutions  issuing from cones, that is, from Lipschitz graphs of positive $1$-homogeneous functions, see Theorem \ref{selfsimilar} and Proposition \ref{prop_backward}.  Moreover,  in  Theorem \ref{convthm2} we show that   Lipschitz graphical evolutions  to \eqref{mcf} asymptotically approach  self-similar expanding solutions, in an appropriate rescaled setting, provided the
initial graph is a  sublinear perturbations of  a cone. More precisely, we introduce the following time rescaling:
 \begin{equation}\label{resc} \tau(t):=\frac{\log(2t+1)}{2}\qquad \text{ and  }\qquad \tilde E_\tau:= \frac{1}{\sqrt{2t+1}} E_t, \end{equation}    
so that  the evolution \eqref{mcf} of the rescaled flow is governed by the geometric law \begin{equation}\label{kflowresc} 
\partial_\tau \tilde p\cdot  \tilde\nu(\tilde p)= -  \tilde p \cdot  \tilde \nu(\tilde p) -  \psi(\tilde \nu(\tilde p) )H_\phi(\tilde p, \tilde E_\tau)\end{equation} 
 and we show that $\tilde E_\tau$ converge locally in Hausdorff sense as $\tau \to +\infty$ to a solution to \eqref{exp}, with $c=1$. 

In the rest of the paper  we analyze   the long time behavior of the flow \eqref{mcf} without  rescaling in the case of regular anisotropies, see assumption \eqref{phiregular} below.  
In order to rule out possible oscillations  in time,  it is necessary to  assume some decay condition of the initial data at infinity, 
and in particular we will consider initial graphs which are asymptotically  flat or asymptotically approaching mean convex cones. 

In Section \ref{sectioncone} we consider self-similar expanding solutions starting from Lipschitz mean convex cones. 
In Theorem \ref{sconvex}, by constructing appropriate barriers, 
we show that such solutions are stable with respect to perturbations vanishing at infinity. 

In Section \ref{sectionhyper} we prove that if the initial surface is asymptotically approaching a hyperplane, then the evolution asymptotically  flattens out in Hausdorff sense. This result is obtained by comparison with large Wulff shapes, and by constructing appropriate $1$-dimensional periodic barriers to the evolution. This approach also provides a different proof for the same result in the isotropic setting, which was  obtained by integral estimates on the flow, see \cites{clutt,e}. 

 
\section{Preliminary definitions and results} \label{secass} 

We recall some definitions for anisotropies and related geometric flows (see for instance  \cite{bccn}).

\begin{definition}\label{defanisotropy}  
Let  $\phi:\R^{N+1}\to [0, +\infty)$ be  a positively $1$-homogeneous convex map,   such that $\phi(p)>0$ for all $p\neq 0$. We associate to the surface tension 
the anisotropy $\phi^0:\R^{N+1}\to [0+\infty)$ defined as $\phi^0(q):=\sup_{\phi(p)\leq 1} p\cdot q$, which is again convex and positively $1$-homogeneous. 
The anisotropic mean curvature of a set $E$ at a point $p\in \partial E$ is defined as
\[H_\phi(p, E)=\div_\tau(\nabla \phi(\nu(p))),  \]  when $\phi$ is regular,  where $\nu(p)$ is the exterior normal vector to $\partial E$ at $p$, and $\div_\tau$ is the tangential divergence, 
whereas  in the general case  it is defined using the subdifferential, 
\[ H_\phi(p,E)\in\div_\tau(\partial \phi(\nu(p))).\] 
\end{definition}

\begin{definition}[$W_{\phi^0}$-condition]\label{ballcondition} 
We define the  Wulff shape as  the convex compact set  \[W_{\phi^0} :=\{q\in\R^{N +1}\ | \phi^0(q) \leq 1\}.\] 
We say that $C\subseteq \R^{N+1}$ satisfies the interior (resp. exterior) $RW_{\phi^0}$-condition at $x\in\partial C$  if 
there exists   $y_x\in \R^{N+1}$ such that $R W_{\phi^0}+y_x\subseteq C$ and $x\in \partial (RW_{\phi^0}+y_x)$ (resp. there exists $y^x$ such that $RW_{\phi^0}+y^x\subseteq \R^{N+1}\setminus C$ 
and $x\in \partial (RW_{\phi^0}+y^x)$).
\end{definition} 

\begin{remark}\upshape 
Observe that if $\phi\in C^{2}(\R^{N+1}\setminus \{0\})$ and $\phi^2$ is uniformly convex, then also $\phi^0\in C^{2}(\R^{N+1}\setminus \{0\})$ and $(\phi^0)^2$ is uniformly convex. In this case the $W_{\phi^0}$-condition is equivalent to the standard (interior or exterior) ball condition. 
\end{remark} 

 We consider the geometric evolution law  \eqref{mcf} under the following assumptions on anisotropy and mobility:
 \begin{equation}\label{asspsi}
\psi:\R^{N+1}\to [0, +\infty) \text{ is  continuous, positively $1$-homogeneous, and }    \psi(p)>0  \qquad \forall p\neq 0
 \end{equation}  and   
  \begin{eqnarray}\label{phiregular} \text{either } &   \text{ $\phi\in C^{2}(\R^{N+1}\setminus \{0\})$ and $\phi^2$ is uniformly convex} \\
   \text{or} &  \text{ $\psi$ is convex.}\label{psinorm} 
 \end{eqnarray} 
 
 \begin{remark}\upshape By positive $1$-homogeneity there holds that $\nabla \phi(\lambda p)=\nabla \phi (p)$ for every $\lambda>0$ and $p\in \R^{N+1}$ and moreover
 that $\nabla \phi(p)\cdot p=\phi(p)$. 
 \end{remark} 

 We associate to the geometric flow \eqref{mcf} the following level set equation: 
 Given a uniformly continuous function $U_0:\R^{N+1}\to\R$ such that $E_0=\{p\in \R^{N+1}:\ U_0(p)\leq 0 \}$  and $\partial E_0=\{p\in \R^{N+1}:\ U_0(p)= 0 \}$,  we consider the solution $U(p,t)$ to  the following quasi-linear parabolic equation 
 \begin{equation}\label{pdelevel}
 \begin{cases} U_t-\psi(\nabla U)\div(\nabla\phi(\nabla U))=0\\U(p,0)=U_0(p).
 \end{cases} 
  \end{equation}
  
\begin{remark}\upshape 
When $\phi$ is sufficiently regular, that is, \eqref{phiregular} holds, the solution to \eqref{pdelevel} is intended in the sense of viscosity solutions, see \cites{bs, gigabook}, whereas in the general case in which $\phi$ is not smooth and $\psi$ is a norm, the solution  is intended as the level set distributional solution defined in \cite{cmnp1}. We recall also that the level set distributional solution is the locally uniform limit of viscosity solutions to \eqref{pdelevel}, when we approximate the anisotropy and the mobility with regular ones, see \cite{cmnp2}.  
 \end{remark} 
 
 We recall the following result about well posedness of the flow \eqref{mcf}. 
 
 \begin{theorem}\label{ex}  
There exists a unique  continuous solution $U$ to \eqref{pdelevel}, to be intended in the sense of viscosity solutions if \eqref{phiregular} holds, and in the distributional level set sense if \eqref{psinorm} holds, that is,  the level set flows defined as 
 \begin{eqnarray*}\label{outin} E^+_t&:=&\{p\in \R^{N+1}:\ U(p,t)\leq 0\}\\E^-_t&:=&\{p\in \R^{N+1}:\ U(p,t)< 0\} \end{eqnarray*}
 provide a solution (in the appropriate sense) to \eqref{mcf}.
 
 Moreover, if $U_0,V_0$ are two uniformly continuous functions such that $U_0\leq V_0$, then $U(p,t)\leq V(p,t)$ for all $t>0$ and $p\in \R^{N+1}$. 
 
Finally, if $U_0$ is Lipschitz continuous with Lipschitz constant $C$, then 
 \[ |U(p,t)-U(q,s)|\leq C |p-q|+C' \sqrt{|t-s|}\qquad \forall p,q\in \R^{N+1}, t,s\geq 0,\] 
 where the constant $C'$ depends on $C$.
 In particular,  if there exists a direction $\omega\in \R^{N+1}$ such that $U_0(p+\lambda\omega)>U_0(p)$ for every $\lambda>0$ and every $p\in \R^{N+1}$, then  $U(p+\lambda\omega,t)>U(p,t)$ for every $t>0$, $\lambda>0$, $p\in \R^{N+1}$.  
 \end{theorem}

\begin{proof} For the existence and  uniqueness of solutions to \eqref{pdelevel}, and the comparison principle, in the case that \eqref{phiregular} holds we refer to \cite{bs}, whereas for the general case in which \eqref{psinorm} holds we refer to \cite{cmnp1}. 

The last two properties are a consequence of the comparison principle and the fact that the differential operator is invariant by translations. Indeed,  if  $U_0$ is Lipschitz continuous, and $C$ is the Lipschitz constant of $U_0$, then for every fixed $h\in \R^{N+1}$, then  $U_h(p, t):=U(p+h,t)\pm C|h|$ is a solution to \eqref{pdelevel} with initial datum $U_0(p+h)\pm C|h|$. Since $U_0(p+h)-C|h|\leq U_0(p)\leq U_0(p+h)+C|h|$, by comparison there holds that $U(p+h,t)-C|h|\leq U(p,t)\leq U(p+h, t)+C|h|$, which implies that 
\[
|U(p,t)-U(q,t)|\leq C |p-q| \qquad \forall p,q\in \R^{N+1}, t> 0.
\]
 A similar argument shows that, if $U_0(p+\lambda\omega)>U_0(p)$ for every $\lambda>0$  then  $U(p+\lambda\omega,t)>U(p,t)$.
 
Finally, observe that if the initial datum is a cone, that is, $V_0(p)=C|p-p_0|$,  for some $p_0\in \R^{N+1}$, then by uniqueness, using the positive $1$-homogeneity of the Euclidean norm and the scaling properties of the operator, we get that the solution to \eqref{pdelevel1} satisfies $V(p-p_0,t)=\frac{1}{r} V(r(p-p_0), r^2 t)$ for every $r>0$ and $t\geq 0, p\in \R^{N+1}$. This implies in particular that \begin{equation}\label{ri} V(0, t)=\sqrt{t}V(0,1)\qquad \text{ for every $t>0$. }\end{equation} 

Therefore, to prove the H\"older continuity of $U$, we proceed as follows. We fix $p_0$ and observe that  $U_0(p)\leq C |p-p_0|+U_0(p_0)$. Hence, by comparison and using \eqref{ri}, we get that $U(p_0,t)\leq \sqrt{t}V(0,1)+U_0(p_0)$, where $V(0,1)$ is the solution to \eqref{pdelevel} with initial datum $C |p-p_0|$. Therefore we get that
$U(p_0,t)-U(p_0, 0)\leq V(0,1)\sqrt{t}.$ By translation invariance of the operator, we conclude that for every $s>0$,  \[U(p_0,t+s)-U(p_0, s)\leq V(0,1)\sqrt{t}.\] 
The other inequality is obtained analogously, taking as initial datum  $-C |p-p_0|$. 
\end{proof} 

 \begin{remark}\label{wulff}\upshape  
 It is easy to check that the rescaled Wulff shape $RW_{\phi^0}$, for $R>0$, satisfies $H_\phi(RW_{\phi^0})=\frac{N}{R}.$ Let  $\underline \psi= \min_{\nu\in\S^{N}} \psi(\nu)>0$, and $\overline \psi=\max_{\nu\in\S^{N}} \psi(\nu)>0$, and define $\underline R(t):=\sqrt{R^2-2\overline{\psi}Nt}$ and $\overline{R}(t):=\sqrt{R^2-2\underline{\psi}Nt}$, for $t$ sufficiently small. Then $\underline R(t)W_{\phi^0}$ is a subsolution to \eqref{mcf} with initial datum $RW_{\phi^0}$, whereas $\overline R(t)W_{\phi^0}$ is a supersolution to \eqref{mcf} with initial datum $RW_{\phi^0}$. This implies that $\underline R(t)W_{\phi^0}\subseteq W^-(t)\subseteq W^+(t)\subseteq \overline R(t)W_{\phi^0}$, where $W^\pm(t)$ is the level set solution to \eqref{mcf}, with initial datum $W_{\phi^0}$ as   defined in Theorem \ref{ex}. 
 \end{remark} 
 
In this paper we consider the case in which   the initial datum $E_0$ is the subgraph of an entire Lipschitz function. Up to a rotation of coordinates, we may assume that 
 \begin{equation}\label{initial} \exists u_0:\R^N\to \R, \text{ Lipschitz continuous such that }E_0=\{(x,z)\in \R^{N+1}\ |\ z\leq u_0(x)\}.
 \end{equation} 
A direct application of  Theorem \ref{ex} gives the following result on the  evolution of Lipschitz graphs. 

 \begin{corollary}\label{corex} 
 Assume that $E_0$ satisfies \eqref{initial}. Then the level set flow satisfies $\overline E_t^-= E_t^+=\{(x,z)\in \R^{N+1}\ |\ z\leq u(x,t)\}$, where $u(x,t)$
 is a continuous function such that  
 \[|u(x,t)-u(y,s)|\leq \|\nabla u_0\|_\infty |x-y|+K\sqrt{|t-s|}\] for some  $K>0$ depending only on the Lipschitz constant $\|\nabla u_0\|_\infty$ of $u_0$. 
 
When $\phi$ is regular, that is, \eqref{phiregular} holds, then $u$ is the viscosity solution to
  \begin{equation}\label{pdelevel1}
 \begin{cases} u_t+\psi(-\nabla u,1)\div(\nabla_x\phi(-\nabla u,1))=0\\u(x,0)=u_0(x).
 \end{cases} 
  \end{equation}
  
When $\phi$ is not regular and \eqref{psinorm} holds, then  the solution is intended  in the distributional sense as in \cite{cmnp1}, 
and coincides with the locally uniform limit of viscosity solutions to \eqref{pdelevel1} when $\phi, \psi$ are approximated by regular functions, see \cite{cmnp2}.  \end{corollary} 

Eventually we recall the following regularity results for solutions to \eqref{pdelevel1}. 

\begin{proposition} \label{reg}  Let $u_0$ be a Lipschitz continuous function. Assume that $\phi$ satisfies \eqref{phiregular},   and that $\nabla^2\phi,\, \psi$ belong to $C^{0,\beta}(\R^{N+1}\setminus \{0\})$ for some $\beta\in (0,1)$. 

Then  $u(\cdot,t)\in C^{2,\alpha}(\R^N)$  and $u(x,\cdot)\in C^{1,\frac\alpha 2}(0,+\infty)$ for every $(x,t)\in \R^N\times (0,+\infty)$
and for some $\alpha\in (0,1)$. Moreover 
  for every $t_0>0$ there exists a positive constant  $C$, depending on $t_0, \phi, \psi$ and the Lipschitz constant of $u_0$, such that 
  \begin{equation}\label{eqreg}
  \|\nabla u(\cdot,t)\|_{C^{1,\alpha}(\R^N)} + \| u_t(x,\cdot)\|_{C^{0,\frac\alpha 2}(t_0,+\infty)} \leq C ,
  \end{equation}
for every $(x,t)\in \R^N\times (t_0,+\infty)$.
 \end{proposition}
\begin{proof}
If the initial datum $u_0$ belongs to $C^{2,\alpha}(\R^N)$, then by \cite[Section 6]{a} (see also \cites{lieb, lunardi}) there exists a solution 
$u:\R^N\times (0,T)\to \R$ for some $T>0$ such that  
$u(\cdot,t)\in C^{2,\alpha}(\R^N)$  and $u(x,\cdot)\in C^{1,\frac\alpha 2}(0,T)$ for every $(x,t)\in \R^N\times (0,T)$.
Then, by standard results  for quasilinear parabolic equations with H\"older continuous coefficients,  see e.g. \cite[Proposition 9.5, Proposition 9.6]{a},  
we have that the norm of the solution $u$ depends only on $\psi, \phi$ and the Lipschitz constant of $u_0$. 
It follows that $T=+\infty$ and \eqref{eqreg} holds.

If $u_0$ is only Lipschitz continuous we approximate $u_0$ with initial data in $C^{2,\alpha}(\R^N)$, and then conclude by 
stability of solutions to \eqref{pdelevel1} with respect to local uniform convergence. 

 \end{proof}
\begin{proposition} \label{reg1} 
Let $u_0\in C^{1,1}(\R^N)$ with   $\|u_0\|_{C^{1,1}(\R^N)}\leq C$ and let \eqref{phiregular} hold.  
 Then the viscosity solution $u$ of \eqref{pdelevel1} is uniformly of class $C^{1,\alpha}(\R^N)$, for any fixed $t>0$ and for all $\alpha\in (0,1)$,
 with  $C^{1,\alpha}$ norm bounded independently of $t>0$.  
 \end{proposition}

\begin{proof}
First of all let $u(\cdot, t)$ be the viscosity solution to \eqref{pdelevel1}. By Corollary  \ref{corex}, it   is Lipschitz continuous with $\|\nabla u(x,t)\|_\infty\leq \|\nabla u_0\|_\infty$.   
  Let $C':= \| \psi(-\nabla u_0,1)\div(\nabla\phi(-\nabla u_0,1))\|_\infty$. 
 Note that $u_0(x)\pm C't$ are respectively  a supersolution and a subsolution to \eqref{pdelevel1}, so that by comparison principle, see Theorem \ref{ex}, we get
 \[u_0(x)-C't\leq u(x,t)\leq u_0(x)+C't\qquad \text{for all $t\geq 0$.}
 \] 
 Moreover for every $t\ge \tau>0$, the functions $u(x, t)\pm \sup_x|u(x, \tau)- u_0(x)|$ are respectively a supersolution and a subsolution to \eqref{pdelevel1} 
 with initial datum $u(x, \tau)$, whence
 $$|u(x, t+\tau)- u(x,t)|\leq \sup_x|u(x, \tau)- u_0(x)|\leq C'\tau .$$ 
 This implies that $u(x, \cdot)$ is Lipschitz continuous with  $|u_t(x, t)|\leq C'$,  which in turn gives, recalling that $\nabla u$ is bounded, that 
 \[ -C\leq  \div(\nabla\phi(-\nabla u,1))\leq C\qquad\text{ for all $x\in \R^N$ and $t>0$.   }\] 
 By elliptic regularity theory for viscosity solutions (see \cite{tru}), this implies that for every $t>0$, $u(\cdot ,t)\in C^{1,\alpha}(\R^N)$ for every $\alpha<1$. \end{proof}

 \section{Self-similar expanding  solutions and convergence of the rescaled flow}\label{sectionhom} 
 We discuss the properties of solutions to \eqref{mcf} starting from Lipschitz cones, that is,   subgraphs of 
 Lipschitz continuous and positively $1$-homogeneous functions $\bar u$:
 \begin{equation}\label{hom}\exists C>0\ \  |\bar u (x)-\bar u(y)|\leq C|x-y|\qquad \bar u(rx)=r\bar u (x)\qquad \forall r>0, x,y\in \R^N.\end{equation} 
Then we consider the long time behavior of solutions   starting from sublinear perturbations of Lipschitz cones. 
 
 \begin{theorem}\label{selfsimilar}Let $\bar u:\R^N\to \R$  be as in \eqref{hom} and let $\bar E_0$ be the subgraph of $\bar u$. 
Then, for every $t>0$,  the evolution  $\bar E_t$ of \eqref{mcf} with initial datum $\bar E_0$  satisfies  for $p\in \partial \bar E_t$ 
 \begin{equation}\label{self}  
 p\cdot \nu=-2 t\, \psi(\nu(p)) H_\phi (p  , \bar E_t),\end{equation}  
 that is,  the flow starting from $\bar E_0$  is an expanding homothetic solution to \eqref{mcf}.
Writing $\bar E_t$  as subgraph of a function $\bar u(\cdot,t)$ we have    that for all $T>0$
\[
  \lim_{t\rightarrow \infty}\bar u(x,t+T)-\bar u(x,t)=0\quad  \text{ 
locally uniformly in $\mathbb{R}^N$}.\] 
Finally, if  either $\bar u\in C^1(\R^N\setminus \{0\})$ or  $\bar E_0$ satisfies the exterior and interior $R_xW_{\phi^0}-$condition at every $x\in\partial\bar E_0\setminus\{0\}$  (see Definition \ref{ballcondition}) with $R_x$ possibly depending on $x$, then  
\[  \lim_{|x|\rightarrow+\infty}\bar u(x,t)-\bar u(x)=0\qquad\text{ for every $t>0$.}\]
\end{theorem} 

\begin{proof}
 By \eqref{hom} we get that for every $r>0$, there holds $r\bar E_0=\bar E_0$, and $H_\phi(r p,r \bar E_0)=r^{-1} H_\phi(p, \bar E_0)$ for all $r>0$ and all $p \in \partial \bar E_0$. Therefore, by uniqueness of solutions and by the rescaling properties of the operator, we get that $\bar E_t= r\bar E_{t/r^2}$ for all $t,r>0$. The previous rescaling identity gives 
\[\bar u(x,t)= \frac{1}{r} \bar u(rx, r^{2} t), \quad r,t>0\quad x \in \mathbb{R}^N.\] 
Letting $r:=t^{-\frac{1}{2}}$ for $t>0$, we get 
\begin{equation} \label{cono}\bar u(x,t)=t^{\frac{1}{2}}\bar u(xt^{-\frac{1}{2}}, 1). \end{equation} This implies that, if $p\in \partial \bar E_1$ then $p t^{\frac{1}{2}}\in \partial \bar E_t$ and  \begin{equation}\label{curv} H_{\phi}(p t^{\frac{1}{2}},\bar E_t)=  t^{-\frac{1}{2}}H_{\phi}(p  , \bar E_1).
\end{equation} 
Substituting in \eqref{mcf}  we get that $ \bar E_{1 }$ solves \eqref{exp} with $c^{-1}= 2$.  The same argument holds substituting $t=1$ with another time $t>0$. 
 
We observe that, by rescaling property \eqref{cono},   
and by the fact that $\bar u(\cdot,t)$ is Lipschitz continuous with the same Lipschitz constant as  $\bar u$, 
for every $T>0$ and $t>0$ we have
\begin{eqnarray*} |\bar u(x,t+T)-\bar u(x,t)|&=&  |(t+T)^{\frac{1}{2}}\bar u(x(t+T)^{-\frac{1}{2}}, 1)-t^{\frac{1}{2}}\bar u(xt^{-\frac{1}{2}}, 1)| \\
&\leq& (t+T)^{\frac{1}{2}}|\bar u(x(t+T)^{-\frac{1}{2}}, 1) - \bar u(x t^{-\frac{1}{2}}, 1)| \\ && + 
|(t+T)^{\frac{1}{2}}-t^{\frac{1}{2}}| |\bar u(xt^{-\frac{1}{2}}, 1)|\\
&\leq & C (t+T)^{\frac{1}{2}}|x| |(t+T)^{-\frac{1}{2}}-t^{-\frac{1}{2}}| + C |(t+T)^{\frac{1}{2}}-t^{\frac{1}{2}}| |x| t^{-\frac{1}{2}}\\
&\leq & C|x| \left(\left(1+\frac{T}{t}\right)^{\frac{1}{2}}- \left(1-\frac{T}{t}\right)^{\frac{1}{2}}\right)\leq \frac{CT|x|}{t}. 
\end{eqnarray*} Sending $t\to +\infty$, we get the result.

We now show that, if $\bar E_0$ satisfies the exterior and interior $W_{\phi^0}-$condition,   the expanding solution is asymptotic at infinity to the initial cone, 
by comparison with the shrinking Wulff shapes constructed in Remark \ref{wulff}.  

Note that by positive $1$-homogeneity of the function $\bar u$,   for every $R>0$, there exists $K>0$ such that  $\bar E_0$ satisfies  the exterior and interior $RW_{\phi^0}-$condition    at every $(x, \bar u(x))\in \partial \bar E_0$, with $|x|\geq K$. Therefore  we get
\[ \bigcup_{|x|\geq K} (RW_{\phi_0}+y_x)\subseteq \bar E_0\subseteq\overline{\R^{N+1}\setminus \bigcup_{|x|\geq K} (RW_{\phi_0}+y^x) }, \] 
where $y_x, y^x$ are introduced in Definition \ref{ballcondition}.  By comparison (see Remark \ref{wulff} and Theorem \ref{ex})
we have  that  \[ \bigcup_{|x|\geq K} (\underline R(t) W_{\phi_0}+y_x)\subseteq  \bar E_t\subseteq\overline{ \R^{N+1}\setminus \bigcup_{|x|\geq K} (\underline R(t) W_{\phi_0}+y^x)}  \] where $\underline R(t):=\sqrt{R^2-2N\overline{\psi}t} $ 
is defined in Remark \ref{wulff}.    From the previous inclusions, we deduce that, for all $|x|\geq K$ and $t< \frac{R^2}{2N\overline{\psi}}$,  
\[|\bar u(x,t)-\bar u(x)|\leq  C(R-\underline{R}(t))\leq C \frac{N\overline{\psi}t}{R}\] where $C$ is a constant which depends on $\phi$ and on the Lipschitz constant of $\bar u$. This implies the conclusion, sending $R\to +\infty$. 
 
A similar argument can be used to prove the same result when $\bar u\in C^1(\R^N\setminus\{0\})$. For it we define  $v_{\lambda}(x):=\bar u(\lambda e+x)-\bar u(\lambda e)$ where $\lambda>0$ and $e\in \mathbb{R}^N$ with $|e|=1$.
For every compact set $K$, since $\nabla \bar u$ is  continuous in $\mathbb{R}^N\setminus \{0\}$, we get that as $\lambda\to +\infty$
\[
\nabla v_{\lambda}(x)=\nabla \bar u(\lambda e+x) =\nabla \bar u\left(e+\frac{x}{\lambda}\right)\rightarrow \nabla \bar u(e)\]  uniformly in $x$ and $e$.
 Hence, letting $f_e(x):=\nabla \bar u(e) \cdot x$ we conclude that  
$$
\lim_{\lambda \rightarrow \infty}\|v_{\lambda}-f_e\|_{C^1(K)}=0
$$
uniformly in $e$, with $|e|=1$. 

This implies that there are functions $\alpha(r), \varepsilon(r)>0$, $r>0$, with $\alpha(r)\rightarrow \infty$ and $\varepsilon(r)\rightarrow 0$ as $r\rightarrow \infty$ so that for large  $|x|$, there exist $y_x, y^x\in \R^{N+1}$ such that  
\[d((x, \bar u(x)),\alpha(|x|)W_{\phi_0}+y_x), d((x, \bar u(x)),\alpha(|x|)W_{\phi_0}+y^x)\leq \varepsilon(|x|),\] 
and eventually
\[(\alpha(|x|)W_{\phi_0}+y_x)\subseteq \bar E_0\subseteq\overline{\R^{N+1}\setminus  (\alpha(|x|)W_{\phi_0}+y^x) }. \]  
So, the thesis follows from the same argument as above.
\end{proof} 
 
\begin{remark}\upshape
\label{remregu} 
If we assume that $\nabla^2\phi, \psi$ belong to $C^{0,\beta}(\R^{N+1}\setminus \{0\})$ for some $\beta\in (0,1)$, then there exists 
 $C>0$ such that   
 \begin{equation}\label{claim11} \sup_{x\in\R^N} |\div(\nabla_x\phi(-\nabla \bar u(x,t),1))|\leq C t^{-\frac{1}{2}}\qquad \forall t>0.
 \end{equation}    
 Indeed  due to  \eqref{curv}   it is sufficient  to check  that there exists $C>0$ such that 
  \[ \sup_{x\in\R^N} |\div(\nabla_x\phi(-\nabla \bar u(x,1),1))|\leq C,\] 
and this  is a consequence of Proposition \ref{reg}. 

Moreover, if \eqref{claim11} is satisfied, then $|\bar u_t(x,t)|\leq C't^{-\frac{1}{2}}$, for some $C'>0$ depending on $C$ 
and on the Lipschitz norm of $u_0$. Integrating in $t$ we get 
\[|\bar u(x, t+T)-\bar u(x,t)|\leq C' (\sqrt{t+T}-\sqrt{t}) \qquad\text{for all $T>0$,}
\]
hence,  for every $T>0$  we conclude that
\[\lim_{t\to +\infty} |\bar u(x,t+T)-\bar u(x,t)|=0\qquad \text{   uniformly in $\R^N$. }\]
\end{remark} 

On the other hand, every homothetically expanding solution $E_t$ to \eqref{mcf} which is the subgraph of a Lipschitz continuous function has a backward in time extension which starts from a subgraph of a suitable Lipschitz and 1-homogeneous function. 
  
 \begin{proposition} \label{prop_backward}
 Assume that $E_1$ is a solution to \eqref{exp} for some $c>0$, such that there exists  a Lipschitz continuous function  $u_1:\mathbb{R}^N\rightarrow \mathbb{R}$  for which
 $E_1=\{(x,z)\ |\ z\leq u_1(x)\}$.  Let $u(x,t)$ be the solution to  \eqref{pdelevel1} in $\R^N\times (1, +\infty)$ with initial datum $u(x,1)=u_1(x)$. 
 Then    we may extend continuously  $u(x,t)$ in  $\R^N\times (t_0, +\infty)$, where  $t_0=1-\frac{1}{2c}$, and moreover it holds 
 $$
 \lim_{t\rightarrow t_0^+}u(x,t)=\bar u(x)\qquad \text{locally uniformly }
 $$
where $\bar u: \mathbb{R}^N\rightarrow \mathbb{R}$ satisfies  \eqref{hom}.
  \end{proposition}
  
 \begin{proof}
Since $E_1$ solves \eqref{exp}, we have that 
\[
 E_t=\lambda(t)E_1, \quad\text{where } \lambda(t):=(2c(t-1)+1)^{\frac{1}{2}}, \quad t\ge 1. 
\] Therefore the solution $u$ to \eqref{pdelevel1} in $\R^N\times (1, +\infty)$ with initial datum $u(x,1)=u_1(x)$ is given by 
\begin{equation}\label{cono_new}
 u(x,t)= \lambda(t)u_1\left(\frac{x}{\lambda(t)}\right), \qquad  t>1, \ x \in \mathbb{R}^N. 
 \end{equation}
By  definition  $\lambda(t)$ is well defined for every $t>t_0=1-\frac{1}{2c}$, and so $u$  can be extended to a continuous function in $\R^N\times (t_0, +\infty)$, which is Lipschitz continuous in $x$, with the same Lipschitz constant as $u_1$. 
 We compute now the limit as $t\to t_0^+$, that is the limit as $\lambda(t)\to 0^+$ in \eqref{cono_new}. For $r\in (0,1)$, we define  \[v_r(x):=ru_1\left(\frac{x}{r}\right).\] Note that $(v_r)_r$  are equi-Lipschitz and locally bounded so that, up to subsequences, there exists the locally uniform limit of $v_r$  as $r\to 0^+$ by Arzel\`a-Ascoli Theorem. 
 
 In order to conclude, we need to show  that such limit   is unique,  i.e., it does not depend on the subsequence.  If this is true, it is easy to check that 
 $\bar u(x):=\lim_{r\rightarrow 0^+}ru_1\left(\frac{x}{r}\right)$ satisfies \eqref{hom}. 
 To prove the whole convergence, we observe that  by   \eqref{cono_new} 
 $$
 v_r(x)= u(x, \lambda^{-1}(r))\qquad \forall r>0.
 $$
 Let   $r_n\rightarrow 0^+$ so that  $\lim_{n\rightarrow \infty}v_{ r_n}(x)=:\bar v(x)$ locally uniformly. We set 
 $$
 u_n(x,t):=u(x, t+\lambda^{-1}(r_n)).
 $$
Then $u_n$ is the solution to \eqref{pdelevel1} in $\R^N\times(0, +\infty)$  with initial datum $u_n(x,0)=u(x,\lambda^{-1}(r_n))=v_{ r_n}(x)$. Recalling that $v_{ r_n}(x)\to \bar v(x)$ locally uniformly, and that $v_{r_n}$ and $\bar v$ are equi-Lipschitz functions, by Corollary \ref{corex}, we get that, up to subsequences, $u_n(x,t)\to \bar v(x,t)$ locally uniformly in $(x,t)$, for some function $\bar v(x,t)$. By stability of solutions to \eqref{pdelevel1}  with respect to uniform convergence, we get that   $\bar v(x,t)$ is the solution to \eqref{pdelevel1} in $R^N\times(0, +\infty)$  with initial datum $\bar v(x,0)=\bar v(x)$.  Now observe that by definition $u_n(x,t)= u(x, t+\lambda^{-1}(r_n))\to u(x, t+t_0)$ for every $t>0$, therefore $\bar v(x,t)=u(x, t+t_0)$ for every $t>0$. This implies that the limit $\bar v$ is   independent of the subsequence $r_n$. 
\end{proof}

We now provide the locally  uniform convergence  to self-similar expanding solutions,  in the rescaled setting \eqref{resc},   
if the initial Lipschitz graph  is a sublinear perturbation of  a cone.  
A similar  result has been obtained  in \cite{eh} for  the isotropic case,  and in \cite{cn} for the fractional mean curvature flow. 
Note that, if the flow $E_t$ is  the subgraph of  the solution $u(x,t)$ to \eqref{pdelevel1}, then the rescaled flow $\tilde E_\tau$ defined in \eqref{resc}
is the   subgraph of   the rescaled function
\begin{equation}\label{rescf} \tilde u(y, \tau):= e^{-\tau}u\left(ye^\tau, \frac{e^{2\tau}-1}{2}\right).\end{equation} 

 \begin{theorem}\label{convthm2}  
   Let $u_0$ be a Lipschitz continuous function,  such that there exist  $\bar u$ which satisfies \eqref{hom}, and constants $K>0$, $\delta\in (0,1)$  for which 
 there    holds\[ |u_0(x)- \bar u(x)|\leq K(1+|x|)^{1-\delta}. \] Let $\tilde u(y, \tau)$  be the rescaled function  as defined in \eqref{rescf}, where $u$ is the solution to 
 \eqref{pdelevel1} with initial data $u_0$ and let $\bar u (x,t)$ be the solution to \eqref{pdelevel1} with initial datum $\bar u$. 
 Then 
\[\lim_{\tau\to +\infty} \tilde u(y, \tau)=\bar u\left(y, \frac{1}{2}\right) \qquad \text{locally uniformly}.\]
 \end{theorem}

\begin{proof} 
We argue as in the proof of \cite[Theorem 5.1]{cn}.   Let $\chi:(0, +\infty)\to (0, +\infty)$ be a smooth function  such that $\chi(k)\equiv 0$ if $k<1$ and $\chi(k)\equiv 1$ if $k>2$.  Define  for $r>1$
\[u_0^r(x)=\bar u(x)+ \chi\left(\frac{|x|}{r}\right)(u_0(x)-\bar  u(x)) . \] 
Then our assumption implies that 
\[|u_0(x)- u_0^r(x)| \leq K(1+2r)^{1-\delta}\qquad \forall x\in \R^N. \] 
By the comparison principle we deduce that 
\[u_r(x,t)-K(1+2r)^{1-\delta}\leq u(x,t)\leq u_r(x,t)+ K(1+2r)^{1-\delta}\qquad \forall x\in \R^N, t>0 \]  where $u, u_r$ are the solutions to \eqref{pdelevel1} respectively with initial data $u_0$ and $u_0^r$. Then passing to the rescaled functions $\tilde u$ and $\tilde u^r$ as defined in \eqref{rescf},  we obtain 
\begin{equation}\label{r1}  \tilde u^r(y,\tau)-K(1+2r)^{1-\delta}e^{-\tau} \leq \tilde u(y, \tau)\leq \tilde u^r(y,\tau)+ K(1+2r)^{1-\delta}e^{-\tau}\qquad \forall y\in \R^N, \tau>0.
\end{equation} 
On the other hand,  
\[\bar u(x)- \frac{2K}{ r^\delta}  |x|\leq u_0^r(x)\leq \bar u(x)+ \frac{2K}{r^\delta} |x|.\]   
Let $\bar u_{\pm r}$ be the solution to \eqref{pdelevel1}  with initial datum respectively $\bar u(x)\pm \frac{2K}{ r^\delta}  |x|$. 
By the comparison principle \begin{equation}\label{scale}  \bar u_{-r}(x,t)\leq u^r(x,t)\leq \bar u_{+r}(x,t)\qquad\forall x\in\R^N, t>0. 
\end{equation} 
Note that $\bar u(x)\pm \frac{2K}{ r^\delta}  |x|$ satisfy \eqref{hom} with Lipschitz constant  $\|D\bar u\|_\infty+ \frac{2K}{ r^\delta}\leq \|D\bar u\|_\infty+2K$.  
By    the rescaling properties of  $\bar u_{\pm r}$, see \eqref{cono}, we get that  when we apply the rescaling \eqref{rescf} to $\bar u_{\pm r}$ we obtain 
$\bar u_{\pm r} \left(y, \frac{1-e^{-2\tau}}{2}\right)$.
Then, passing in \eqref{scale} to the rescaled functions as defined in \eqref{rescf} we get 
\[\bar u_{-r} \left(y, \frac{1-e^{-2\tau}}{2}\right)\leq \tilde u^r(y, \tau) \leq\bar u_{+r} \left(y, \frac{1-e^{-2\tau}}{2}\right). \] 
 Recalling that $\bar u(x)\pm \frac{2K}{ r^\delta}  |x|$ are Lipschitz functions with Lipschitz constant bounded by  $ \|D\bar u\|_\infty+2K$, 
 by Corollary  \ref{corex}
 we get that there exists $B$ depending only on $\|D\bar u\|_\infty$ and $K$ such that 
\[\bar u_{-r}\left(y, \frac{1}{2} \right)-Be^{-\tau} \leq \tilde u^r(y, \tau) \leq \bar u_{+r}\left(y, \frac{1}{2} \right)+Be^{-\tau}. \] 
Therefore by \eqref{r1} we conclude that
\[  \bar u_{-r}\left(y, \frac{1}{2} \right)-(B+K(1+2r)^{1-\delta})e^{-\tau} \leq \tilde u(y, \tau)\leq \bar u_{+r}\left(y, \frac{1}{2} \right)+(B+K(1+2r)^{1-\delta})e^{-\tau}, \] 
 for all $y\in \R^N$, $\tau>0$ and $r>1$.
 
Notice that $\bar u_{\pm r}\left(y, \frac{1}{2} \right)\to  \bar u\left(y, \frac{1}{2} \right)$ 
as $r\to +\infty$,  
locally uniformly in $y$ by stability of solutions with respect to local uniform convergence, since $\bar u(x)\pm \frac{2K}{r^\delta} |x|\to \bar u(x)$ locally uniformly. 
Therefore, taking $r=e^\tau$ in the previous inequality and letting $\tau\to +\infty$, we obtain the local uniform convergence of $\tilde u$. 
\end{proof}  

 \begin{remark}\upshape 
 If $u_0-\bar u\in L^\infty(\R^N)$, the convergence result in Theorem \ref{convthm2} can  be strengthened to  uniform convergence. 
 In the isotropic case, the uniform convergence has been obtained in \cite{eh} under the assumptions of Theorem \ref{convthm2}, by using maximum principle and integral estimates. 
 \end{remark}

\section{Stability of self-similar solutions asymptotic to mean convex cones}\label{sectioncone} 
 In this section we assume that the anisotropy is regular, that is, \eqref{phiregular} holds, and we address the issue of the stability with respect to perturbations vanishing at infinity in the case of mean convex cones.   The same problem  has been considered in the isotropic setting in \cite{clutt}. 
 
So we consider self-similar expanding solutions starting from initial data which satisfy the following condition: $\bar u$ is as in  \eqref{hom} and  moreover
\begin{equation}\label{convex} \bar u\in C^2(\R^{N}\setminus \{0\}),  \text{  $\bar u$ is nonlinear  and  } \div(\nabla_x\phi(-\nabla \bar u,1))< 0\text{ in the viscosity sense}.
\end{equation} 
Note that the assumption in \eqref{convex} implies  that the epigraph of $\bar u$, that is, the set $\{(x,z)\ |\ z\geq \bar u(x)\}$ is a mean convex set. 
 
We first show that mean convexity is preserved and that the homothetic solution $\bar u(x,t)$ always lies above $\bar u$. 
\begin{lemma} \label{lemmapos} Let $\bar u$ be as in \eqref{hom} and  \eqref{convex} and let $\bar u(x,t)$ be the solution to \eqref{pdelevel1} with initial datum $\bar u$. Then 
$\bar u(x,t+s)\geq \bar u(x,t)$ for every $t\geq 0$ and $s>0$, and moreover $\bar u(0,t)=\sqrt{t}\bar u(0,1)>0$ for $t>0$.  \end{lemma} 
\begin{proof} Condition \eqref{convex} implies that $\bar u(x)$ is a stationary subsolution to \eqref{pdelevel1}, so by comparison $\bar u(x,t)\geq \bar u(x)$. Again by comparison and the semigroup property we get that for all $s>0$, there holds $\bar u(x, t+s)\geq \bar u(x,t)$.

In particular we get that $\bar u(0,t)\geq \bar u(0)=0$, moreover by \eqref{cono}, $\bar u(0,t)=\sqrt{t} \bar u(0,1)$. Observe that since $\bar u(0,1)$ 
is a solution to \eqref{exp} with $c=1/2$, then $ \div(\nabla_x\phi(-\nabla \bar u,1))\in L^\infty_{loc}$ and by elliptic regularity theory \cite{tru}, recalling \eqref{phiregular}, 
this implies that  $\bar u(\cdot ,1)\in C^{1,\alpha}(\R^N)$  for every $\alpha<1$. It follows that $\bar u(0,1)>0$. 
\end{proof} 

 \begin{proposition}\label{lemmabarr}  Let $\bar u$ be as in \eqref{hom} and  \eqref{convex} and let $u_0:\R^N\to \R$ be a Lipschitz continuous function such that 
  \[   \lim_{|x|\to +\infty} u_0(x)-\bar u(x)=0. \] 
  Then for every $\delta>0$ there exists $t_\delta>0$ such that
  \[u(x,t)\geq \bar u(x)-\delta, \qquad t\geq t_\delta\] where $u(x,t)$ is the viscosity solution to \eqref{pdelevel1} with initial datum $u_0$. 
  \end{proposition}
  \begin{proof}The proof is based on the construction of a barrier for the evolution. 
   First of all observe that by assumption, there exists $m>0$ such that $u_0(x)\geq \bar u(x)-m$ for all $x\in\R^N$. Moreover, for every $\delta>0$ there exists  $R(\delta)>0$ such that $u_0(x)\geq \bar u(x)-\delta$ for every $|x|\geq R(\delta)$.
     
 Let  $\chi:[0, +\infty)\to [0, 1]$ be a  smooth function such that $\chi(s)=1$ if $s<1$, $\chi(s)=0$ if $s>2$.  
Then  for every $r>0$, we define $b_r(x):=\bar u(x)- r\chi(|x|)$. Then $b_r\in C^2(\R^N\setminus\{0\})$, moreover $b_r(x)\leq \bar u(x)$ for all $x$, in particular $b_r(x)=\bar u(x)$
 if $|x|>2$ and $b_r(x)=\bar u(x)-r$ if $|x|\leq 1$.  Finally  we get 
\[\lim_{r\to 0}  \max_{|x|\in [1,2]} \div(\nabla_x\phi(-\nabla b_r(x),1))\leq \max_{|x|\in [1,2]} \div(\nabla_x\phi(-\nabla \bar u,1))=-c<0.\]
Choosing $r$ sufficiently small, we get that $\div(\nabla_x\phi(-\nabla b_r(x),1))<0$ in the viscosity sense for all $x$. Now, let us fix $r_0$, and $b_{r_0}$ the corresponding function such that the previous condition is satisfied. In particular $b_{r_0}$ is a stationary subsolution to \eqref{pdelevel1}. 

 We claim that there exists $\lambda>0$ such that
 \begin{equation}\label{claim1} u(x,t)\geq w(x,t):=\sup\left(\bar u(x,t)-m,  \lambda b_{r_0}\left(\frac{x}{\lambda}\right)-\delta \right)  \end{equation} 
where $\bar u(x,t)$ is the solution to \eqref{pdelevel1} with initial datum $\bar u$. Note that  the function
 $w(x,t)$ is a subsolution to \eqref{pdelevel1}, since it is the supremum between two subsolutions. 
 Therefore to check the claim, it is sufficient to show that $u(x,0)=u_0(x)\geq w(x,0)=\sup\left(\bar u(x)-m,   \lambda b_{r_0}\left(\frac{x}{\lambda}\right)-\delta \right)$. First of all, by assumption we know that $u_0(x)\geq \bar u(x)-m$ for all $x$. On the other hand, observe that if $\lambda>\max(R(\delta), m/r_0)$, by definition of $\chi$ and the positive $1$-homogeneity of $\bar u$, there holds that  for $|x|\leq R(\delta)<\lambda$, 
\[\lambda b_{r_0}\left(\frac{x}{\lambda}\right)=\bar u(x)-\lambda r_0 \chi \left(\frac{|x|}{\lambda}\right)=\bar u(x)-\lambda r_0\leq\bar u(x)-m\leq u_0(x).\]  On the other hand if $|x|\geq R(\delta)$, by assumption and our construction of $b_r$,  $u_0(x)\geq \bar u(x)-\delta\geq \lambda b_{r_0}\left(\frac{x}{\lambda}\right)-\delta$. 

Therefore, by comparison, \eqref{claim1} holds for every $\lambda>\max(R(\delta), m/r_0)$. We fix $\lambda_0$ which satisfies this condition. 

Now, observe that by Lipschitz continuity and Lemma \ref{lemmapos}, there holds  $\bar u(x,t)-m\geq \sqrt{t} \bar u(0,1)-\|\nabla \bar u\|_\infty|x| -m$. 
Therefore   there exists $t_\delta$ such that $  \bar u(x,t)-m\geq \bar u(x) $, for all $|x|\leq 2\lambda_0$ and for all $t\geq t_\delta$ and 
then, in turn, by \eqref{claim1}, we get that  $u(x,t)\geq \bar u(x)$ for all $t\geq t_\delta$ and $|x|
\leq 2\lambda_0$. 
On the other hand, if $|x|>2\lambda_0$  there holds by definition that $\lambda_0 b_{r_0}\left(\frac{x}{\lambda_0}\right)-\delta=\bar u(x)-\delta$ and then again by \eqref{claim1},
$u(x,t)\geq \bar u(x)-\delta$ for all $t$, for all $|x|>2\lambda_0$. 
So, we get the conclusion.
\end{proof} 
We conclude with the following stability result. 

 \begin{theorem} \label{sconvex} Assume that \eqref{phiregular} holds.  Let $u_0:\R^N\to \R$ be a Lipschitz continuous function and    $\bar u$ a nonlinear function which satisfies \eqref{hom},  and  \eqref{convex}
 such that  
 \[   \lim_{|x|\to +\infty} u_0(x)-\bar u(x)=0. \] 
Then, 
\[\lim_{t\to +\infty} u(x,t)-\bar u(x,t)=0\qquad\text{locally uniformly in $\R^N$}\] where $u(x,t), \bar u(x,t)$ are the solutions to \eqref{pdelevel1} with initial datum $u_0, \bar u$. 
If moreover $\nabla^2\phi, \psi$ belong to $C^{0,\beta}(\R^{N+1}\setminus \{0\})$ for some $\beta\in (0,1)$,  then the convergence is uniform in $\R^N$. 
\end{theorem}
\begin{proof}
First of all we observe that by  Proposition \ref{lemmabarr} and the comparison principle,  there holds that for all $\delta>0$ there exists $t_\delta$ such that for all $t>t_\delta$, 
\begin{equation}\label{out}u(x,t)\geq \bar u(x, t-t_\delta)-\delta.\end{equation}

We fix now $\eps>0$ and $R_\eps>0$ such that $u_0(x)\leq \bar u(x)+\eps$ for all $|x|>R_\eps$. Therefore, by Lemma \ref{lemmapos},  we get that
for all $t>0$,  \begin{equation}\label{due}u_0(x)\leq  \bar u(x)+\eps\leq \bar u(x,t)+\eps\qquad \text{ for all $|x|>R_\eps$}.\end{equation} 
Observe now that  by \eqref{cono} and Lipschitz continuity \[\bar u(x,t)\geq \bar u(0, t)- \|\nabla\bar u\|_\infty |x|= \sqrt{t}\bar u(0,1)-\|\nabla\bar u\|_\infty |x|.\] 
Since $\bar u(0,1)>0$ by Lemma \ref{lemmapos}, 
 there exists $t_\eps>0$ sufficiently large such that \begin{equation}\label{uno}\bar u(x,t)\geq u_0(x)\qquad\text{ for all $|x|\leq R_\eps$, $t\geq t_\eps$. }\end{equation}

Therefore, by \eqref{due}, \eqref{uno}, and by comparison we get that for every $t>0$
\begin{equation}\label{in}u(x,t)\leq \bar u(x, t+t_\eps)+\eps\qquad \forall t\geq 0, x\in \R^N.\end{equation}

By \eqref{in} and \eqref{out} we get that for $\eps>0$, $\delta>0$ 
\[ \bar u(x, t-t_\delta)-\bar u(x,t)-\delta  \leq u(x,t)-\bar u(x,t)\leq  \bar u(x, t+t_\eps)-\bar u(x,t)+\eps.\]
We conclude sending $t\to +\infty$, and recalling Theorem \ref{selfsimilar}. In case that  $\nabla^2\phi, \psi$ are more regular, we conclude recalling Remark \ref{remregu}. 
\end{proof} 

  \section{Stability of hyperplanes}\label{sectionhyper} 
In this section we   show that hyperplanes are stable with respect to the flow \eqref{mcf},  if the anisotropy is regular, that is, \eqref{phiregular} holds. In particular we show that 
 if the initial datum is flat at infinity, then the solution stabilizes  to the hyperplane  at which the initial datum is asymptotic.  The same result has been obtained in the isotropic case by integral estimates and comparison with large balls in \cites{clutt, e}, and by using the heat kernel, in dimension $2$, in \cite{n}. Here we provide a different proof based on construction of suitable periodic barriers. 
 
  We start with a preliminary lemma on $1$-dimensional periodic barriers.
  
\begin{lemma}\label{lemmaper}
Assume \eqref{phiregular}, and
let $f_0:\R\to \R$ be a   function  of class $C^{1,1}(\R)$ which is
$\Z$-periodic and odd.  Let $u(x,t)$ be the solution to \eqref{pdelevel1} with initial datum $u_0(x):=f_0(x\cdot e_1)$.  Then 
\[\lim_{t\to +\infty} u(x, t) =0 \qquad\text{uniformly in $C(\R^N)$. }\]   
\end{lemma} 

 \begin{proof} 
 By uniqueness we have that \[u(x,t)= f(x\cdot e_1, t),\] where $f:\R\times[0, +\infty)\rightarrow \R$ satisfies $f(r,0)=f_0(r)$, $f(r, t)= -f(-r,t)$ and $f(r+z, t)=f(r,t)$  for every $z\in \Z$. 
 Notice also that 
 \begin{equation}\label{lip}
 |\nabla u(x,t)|\le \max_{[0,1]} f_0'\qquad\text{for all $(x,t)\in \R^N\times (0,+\infty)$.}
 \end{equation}
 Differentiating in time the anisotropic perimeter of the graph of $u$, for all $t>0$ we get that
 \begin{eqnarray*} 
 \int_{[0,1]^{N}}  \phi(-\nabla u_0(x), 1) dx &\ge&
 \int_{[0,1]^{N}}  \phi(-\nabla u_0(x), 1) dx-\int_{[0,1]^{N}}  \phi(-\nabla u(x,t), 1) dx \\
  &=&
 - \int_0^t\int_{[0,1]^N}u_t\ \div(\nabla \phi((-\nabla u, 1)dxds
 \\&=&\int_0^t\int_{[0,1]^N}\frac{u_t^2}{\psi(-\nabla u,1)}\,dxds,
 \end{eqnarray*} 
 which implies that 
 \[
 \int_0^{+\infty}\int_0^1 f_t^2\, dxds = \int_0^{+\infty}\int_{[0,1]^N} u_t^2\, dxds \le C,
 \]
 where the constant
 \[
 C:= \left(\max_{|\xi|\le \max_{[0,1]} f_0'} \psi(\xi,1)\right) \int_{[0,1]^{N}}  \phi(-\nabla u_0(x), 1) dx
 \]
 depends only on the initial function $f_0$. It follows that there exists a sequence of times $t_n\to +\infty$ such that 
 $f_t(\cdot,t_n)\to 0$ in $L^2([0,1])$ as $n\to +\infty$. Recalling Proposition \ref{reg1},
 up to extracting a further subsequence we can also assume that $u(\cdot,t_n)\to \bar u(x)$ uniformly
 in $C^1(\R^N)$ as $n\to +\infty$, where $\bar u(x) = \bar f(x\cdot e_1)$ with $\bar f$  $\Z$-periodic.

 Evaluating now \eqref{pdelevel1} at $t=t_n$, after an integration by parts we get that 
 \[
 \int_{[0,1]^N} \frac{u_t(x,t_n)}{\psi(-\nabla u(x,t_n),1)}\,\eta(x\cdot e_1) dx 
 =   \int_{[0,1]^N} \nabla \phi(-\nabla u(x,t_n),1)\cdot e_1\,\eta'(x\cdot e_1) dx
 \]
 for all $\eta\in C^1(\R)$ $\Z$-periodic.
 Passing to the limit as $n\to +\infty$,
 we finally get that the function $\bar u$ satisfies
 \[
  \int_{[0,1]^N} \nabla_x \phi(-\nabla \bar u(x),1)\cdot e_1\,\eta'(x\cdot e_1) dx = 0
 \]
 for all $\eta$, that is,  $\bar u$ is a periodic, odd and smooth solution to the anisotropic minimal surface equation
 \[
 {\rm div}(\nabla_x\phi(-\nabla \bar u,1))=0.
 \]
Recalling that we are assuming \eqref{phiregular}, we get that  by the strong maximum principle, and the periodicity of $\bar u$, we conclude that  $\bar u\equiv 0$.\end{proof}

\begin{theorem}\label{stabplanes} Let $E_0\subseteq \R^{N+1}$  such that  $\partial E_0$ is a Lipschitz surface and assume  that there exists a half-space $H$ for which
\[\lim_{R\to +\infty} d_H(E_0\setminus B(0, R), H\setminus B(0, R))=0\] where  $d_H(A,B)$  is the Hausdorff distance between the sets $A,B$. Then we have  
\[\lim_{t\to +\infty}d_H(E_t, H)=0. \]\end{theorem} 

\begin{proof} 
We may assume without loss of generality that $H=\{(x, z)\in \R^{n}\times\R\ |\ z\leq 0\}$. Moreover, by the assumption that $\lim_{R\to +\infty} d(E_0\setminus B(0, R), H\setminus B(0, R))=0$, there exist two Lipschitz functions $u_0, v_0:\R^N\to\R$ such that $\lim_{|x|\to +\infty} u_0(x)=0=\lim_{|x|\to +\infty} v_0(x)$ and  $\{(x, z)\in \R^{n}\times\R\ |\ z\leq u_0(x)\}\subseteq E_0\subseteq \{(x, z)\in \R^{n}\times\R\ |\ z\leq v_0(x)\}$. By  comparison we get that 
$\{(x, z)\in \R^{n}\times\R\ |\ z\leq u(x,t)\}\subseteq E_t\subseteq \{(x, z)\in \R^{n}\times\R\ |\ z\leq v(x,t)\}$, where $u(x,t), v(x,t)$ are the solutions to \eqref{pdelevel1} with initial datum $u_0, v_0$. 

We claim that $\lim_{t\to +\infty} u(x,t)=0=\lim_{t\to +\infty} v(x,t)$ uniformly. If the claim is true, then the result follows. 

It is sufficient to prove the claim only for $u$ (since for $v$ is completely analogous), moreover, we  restrict to the case in which $u_0\geq 0$ (or equivalenty $u_0\leq 0$). Indeed the general case is easily obtained by using as barriers  the solutions with initial data $u_0^+=\max (u_0, 0)$ and $u_0^-=\min(u_0, 0)$.  

So, we prove the claim only for the case  $u_0\geq 0$.  
Note that by comparison, since the constants are stationary solutions, $0\leq u(x,t)\leq \max u_0$ for all $x\in \R^N, t>0$.    

First of all we prove that 
\[\inf_x u(x,t)=0\] for all $t>0$, by comparison with shrinking Wulff shapes as constructed in Remark \ref{wulff}.
Indeed for every $\eps>0$, let us fix $R>0$ such that $u_0(x)\leq \eps$ for all $|x|\geq R$.  
Let $\tilde E_0$ be the subgraph of $\max (u_0, \eps)$. Note that   for $|x|\geq R$, $\partial \tilde E_0$ is a hyperplane, and then at every    $|x|>R$,   there exists  $R_{|x|}>0$, such that $R_{|x|}\to +\infty$ as $|x|\to +\infty$ and that $\tilde E_0$ satisfies at $x$ the exterior $R_{|x|}W_{\phi^0}$ condition.   
So, arguing as in the proof of Theorem \ref{selfsimilar}, we get that for every $t>0$, there exists $K>0$ such that $|u(x,t)|\leq 2\eps$ for $|x|>K$. 
This gives the desired property  $\inf _xu(x,t)=0$.

Now we note that by comparison, $M(t)=\sup_{x} u(x,t)$ is decreasing. We define $0\leq \bar M:=\lim_{t\to +\infty} M(t)=\inf_t M(t)$.  Now we claim that \[\bar M=0.\] 
If the claim holds, then we get the conclusion.   

 Assume by contradiction that $\bar M>0$. We fix $0<\eps<\frac{\bar M}{2}$ and $\bar t>0$ such that $M(\bar t)\leq \bar M+\eps$. We fix also $R=R(\bar t)$ such that $u(x, \bar t)<\frac{\bar M}{2}$ for all  $|x|>R$.   
Now we shall get a contradiction by constructing a  periodic barrier  as in Lemma \ref{lemmaper} (up to suitable vertical  translations).
We fix a  smooth even  function $f_0:[-2R, 2R] \to\R$, such that  $f_0(z)=\bar M+\eps$ for $z\in[-R, R]$, $f_0(-2R)=f_0(2R)=  \frac{3}{4}\bar M+\frac{\eps}{2}$ and $f_0(z)$ is increasing in $(-2R, -R)$ and decreasing in $(R,2R)$. 
 Now we extend it to a function $f_0:[-2R, 6R]\to \R$ by putting \[f_0(z+2R)=-f_0(-z+2R)+\frac{3}{2}\bar M+\eps.\] Note that $f_0(z+2R)- \frac{3}{4}\bar M-\frac{\eps}{2}$ is an odd function. Finally, we extend it by periodicity  to be a $8R\Z$ periodic function. 

Let $v_0(x):=f_0(x\cdot e_1)$, and 
 note that by construction, $u(x, \bar t)\leq v_0(x)$ for all $x\in \R^N$ and then by comparison
  \[u(x,t+\bar t)\leq v(x,t),\qquad \text{and in particular }  \limsup_{t\to +\infty} u(x,t)\leq \lim_{t\to +\infty} v(x,t)\] where $v(x,t)$ is the solution to \eqref{pdelevel1} with initial datum $v_0$. 
  Now by Lemma \ref{lemmaper} we get that $\lim_{t\to +\infty}v(x,t)=\frac{3}{4}\bar M+\frac{\eps}{2}$  uniformly.  Since $\frac{3}{4}\bar M+\frac{\eps}{2}<\bar M$,  by our choice of $\eps$, 
we get that  $\limsup_{t\to +\infty} u(x,t)<\bar M$, in contradiction  with the definition of $\bar M$. 
 \end{proof}


  \end{document}